\pdfoutput=1
\documentclass[12pt,a4paper,reqno]{amsart}
\usepackage{amssymb}

\usepackage{amscd}
\usepackage{enumerate}
\usepackage{graphicx}
\usepackage{siunitx}
\usepackage{tikz-cd}
\usepackage{color}
\usetikzlibrary{arrows}
\numberwithin{equation}{section}

\usepackage{mathabx}

\usepackage{mathtools}
\usepackage[tableposition=top]{caption}
\usepackage{booktabs,dcolumn}

\DeclareCaptionLabelSeparator{separation}{:\quad}
\DeclareFontFamily{OT1}{rsfs}{}
\DeclareFontShape{OT1}{rsfs}{n}{it}{<-> rsfs10}{}
\DeclareMathAlphabet{\mathscr}{OT1}{rsfs}{n}{it}

\theoremstyle{plain}

\newtheorem{theorem}{Theorem}[section]
\newtheorem{proposition}[theorem]{Proposition}
\newtheorem{lemma}[theorem]{Lemma}
\newtheorem{corollary}[theorem]{Corollary}
\newtheorem{conjecture}[theorem]{Conjecture}

\theoremstyle{definition}

\newtheorem{remark}[theorem]{Remark}

\newcommand\R{\mathbb{R}}
\newcommand\Z{\mathbb{Z}}
\newcommand\N{\mathbb{N}}
\newcommand\C{\mathbb{C}}

\newcommand\eps{\varepsilon}

\renewcommand{\sim}{\asymp} 

\renewcommand{\phi}{\varphi}

\begin{document}

\title[Elliott-Halberstam implies Vinogradov]{The Elliott-Halberstam conjecture implies the Vinogradov least quadratic nonresidue conjecture}

\author{Terence Tao}
\address{Department of Mathematics, UCLA\\
405 Hilgard Ave\\
Los Angeles CA 90095\\
USA}
\email{tao@math.ucla.edu}

\begin{abstract}  For each prime $p$, let $n(p)$ denote the least quadratic nonresidue modulo $p$.  Vinogradov conjectured that $n(p) = O(p^\eps)$ for every fixed $\eps>0$.  This conjecture follows from the generalised Riemann hypothesis, and is known to hold for almost all primes $p$ but remains open in general.  In this paper we show that Vinogradov's conjecture also follows from the Elliott-Halberstam conjecture on the distribution of primes in arithmetic progressions, thus providing a potential ``non-multiplicative'' route to the Vinogradov conjecture.  We also give a variant of this argument that obtains bounds on short centred character sums from ``Type II'' estimates of the type introduced recently by Zhang and improved upon by the Polymath project, or from bounds on the level of distribution on variants of the higher order divisor function.  In particular, we can obtain an improvement over the Burgess bound would be obtained if one had Type II estimates with level of distribution above $2/3$ (when the conductor is not cube-free) or $3/4$ (if the conductor is cube-free); morally, one would also obtain such a gain if one had distributional estimates on the third or fourth divisor functions $\tau_3, \tau_4$ at level above $2/3$ or $3/4$ respectively.   Some applications to the least primitive root are also given.
\end{abstract}

\maketitle


\section{Introduction}

For each prime $p$, let $n(p)$ denote the least natural number that is not a quadratic residue modulo $p$.  Vinogradov \cite{vin} established the asymptotic bound
\begin{equation}\label{vino}
n(p) \ll p^{\frac{1}{2\sqrt{e}}} \log^2 p
\end{equation}
for all primes $p$, and made the following conjecture:

\begin{conjecture}[Vinogradov's conjecture]\label{vinconj}  For any fixed $\eps>0$, we have $n(p) \ll p^\eps$.
\end{conjecture}

(See Section \ref{notation} below for our conventions on asymptotic notation.)  Linnik \cite{linnik} showed that this conjecture follows\footnote{In fact, the conjecture follows from even very weak fragments of this hypothesis; see e.g. \cite[Theorem 10.6]{bateman}.  (Thanks to Kevin Ford for this reference.)  The strongest result in this direction comes from a very recent work of Granville and Soundararajan \cite{gs-new} (see also \cite{banks}), who showed (roughly speaking) that the only way this conjecture can fail is if a positive proportion of low-lying zeroes of an $L$-function lie extremely close to the line $\operatorname{Re}(s)=1$.} from the generalised Riemann hypothesis; Ankeny \cite{ankeny} improved the bound further to
$$ n(p) \ll \log^2 p$$
on this hypothesis.
However, Conjecture \ref{vinconj} remains open unconditionally; the best bound available (up to logarithmic factors) for general primes $p$ is
\begin{equation}\label{burgess}
n(p) \ll p^{\frac{1}{4\sqrt{e}}+\eps}
\end{equation}
for any fixed $\eps>0$, a well-known result of Burgess \cite{burgess}.  It was also shown by Linnik \cite{linnik} unconditionally that for any fixed $\eps>0$, the number of $p \leq x$ with $n(p) > x^\eps$ is bounded uniformly in $x$, and hence the number of exceptions to the inequality $n(p) > p^\eps$ with $p \leq x$ is bounded by $O( \log \log x )$.

In this paper we connect Vinogradov's conjecture to a standard conjecture in sieve theory, the \emph{Elliott-Halberstam conjecture} \cite{elliott}, as well as to a restricted fragment of this conjecture recently introduced by Zhang \cite{zhang}.  The basic phenomenon being exploited here is that distribution estimates such as those given by the Elliott-Halberstam conjecture allow one to control correlations of the form\footnote{If only the original Elliott-Halberstam conjecture is available, rather than its variants, then one of the convolutions $\alpha*\beta$ or $\gamma*\delta$ needs to be replaced by the von Mangoldt function $\Lambda$.  Also, for technical reasons it is convenient to ensure that one of the factors $\alpha,\beta,\gamma,\delta$ is supported on numbers coprime to the shift $h$.}
\begin{equation}\label{abgd}
 \sum_n (\alpha * \beta)(n) (\gamma * \delta)(n+h)
\end{equation}
for various arithmetic sequences $\alpha,\beta,\gamma,\delta$ and non-trivial shifts $h$, as long as all of the sequences $\alpha,\beta,\gamma,\delta$ vanish for very small values of $n$, and provided that at least one of the sequences $\alpha,\beta,\gamma,\delta$ is ``smooth'' (e.g. if one of these sequences is an indicator function such as $1_{[N,2N]}$).  On the other hand, by combining the multiplicativity and periodicity properties of Dirichlet characters with a hypothesis that the least quadratic residue is large (or that a character sum is large), we will be able to construct sums of the form \eqref{abgd} that deviate substantially from its expected value, giving the required contradiction.  It is the periodicity of Dirichlet characters $\chi$ that allow us to introduce the shift $h$, thus transferring the problem from a multiplicative number theory problem (in which hypotheses such as the generalised Riemann hypothesis are useful) to a sieve theory problem (in which hypotheses such as the Elliott-Halberstam conjecture are useful).  The arguments share some similarities with that of Burgess \cite{burgess} (which also relies heavily on the multiplicativity and periodicity properties of Dirichlet characters), but is ultimately powered by a somewhat different source of cancellation, namely the equidistribution assumptions of Elliott-Halberstam type, rather\footnote{It is worth noting however that much of the recent partial progress on the Elliott-Halberstam conjecture has proceeded by using Weil exponential sum estimates, although the precise estimates used there are different from those used in the Burgess argument.  In Section \ref{variant-sec}, though, we sketch a version of the argument that allows for an improvement over the original bound \eqref{vino} of Vinogradov using only the elementary bound of Kloosterman \cite{kloost} on Kloosterman sums, and does not require the full strength of the Weil conjectures.} than the Weil exponential sum estimates.  

To describe the results more precisely we need some notation.
For any function $\alpha \colon \N \to \C$ with finite support (that is, $\alpha$ is non-zero only on a finite set) and any
primitive residue class $a\ (r)$, we define the (signed)
\emph{discrepancy} $\Delta(\alpha; a\ (r))$ to be the quantity
\begin{equation}\label{disc-def}
  \Delta(\alpha; a\ (r)) \coloneqq \sum_{n = a\ (r)} \alpha(n) - 
  \frac{1}{\phi(r)} \sum_{(n,r)=1} \alpha(n)
\end{equation}
where $\phi$ is the Euler totient function.

\begin{conjecture}[Elliott-Halberstam conjecture]\label{ehj}  Let $0 < \vartheta < 1$ be fixed.  Then one has
\begin{equation}\label{vrt}
\sum_{r < x^\vartheta} \sup_{a \in (\Z/r\Z)^\times} | \Delta( \Lambda 1_{[1,x]}; a\ (r) ) | \ll x \log^{-A} x
\end{equation}
for any fixed $A > 1$, where $\Lambda$ is the von Mangoldt function.  Equivalently, from the prime number theorem, one has
$$
\sum_{r < x^\vartheta} \sup_{a \in (\Z/r\Z)^\times} \left| \sum_{n \leq x: n = a\ (r)} \Lambda(n) - \frac{x}{\phi(r)} \right| \ll x \log^{-A} x
$$
for any fixed $A > 1$.
\end{conjecture}

The case $\vartheta < 1/2$ of this conjecture is of course (a slightly weakened form of) the Bombieri-Vinogradov theorem \cite{bombieri,vinogradov}.

Our first theorem is then

\begin{theorem}[Elliott-Halberstam implies Vinogradov]\label{eh1}  Conjecture \ref{ehj} implies Conjecture \ref{vinconj}.
\end{theorem}

We prove this theorem in Section \ref{eh-sec}.  The basic idea is to observe (from the general theory of mean values of multiplicative functions) that if $n(q) > q^\eps$ for some large prime $q$, then the character sum $\sum_{n \leq x} \chi(n) \Lambda(n)$ will be anomalously large for some large $x = O(q^{O(1)})$, where $\chi$ is the quadratic character modulo $q$.  As $\chi$ is periodic modulo $q$, this forces $\sum_{n \leq x} \chi(n) \Lambda(n+q)$ to be large also.  But one can use the Elliott-Halberstam conjecture (and an expansion of $\chi$ into divisor sums, using once again the largeness of $n(q)$) to obtain good bounds for $\sum_{n \leq x} \chi(n) \Lambda(n+q)$ and obtain a contradiction.

With some additional combinatorial argument, we can obtain a similar implication\footnote{We are indebted to Felipe Voloch for suggesting this variant.} concerning the least primitive root modulo $p$, provided that $p-1$ has only boundedly many factors:

\begin{theorem}[Elliott-Halberstam bounds least primitive roots]\label{eh-a}  Assume Conjecture \ref{ehj}.  Then for any fixed $d \geq 1$ and fixed $\eps > 0$, and any prime $p$ for which $p-1$ is the product of at most $d$ primes (counting multiplicity), the least primitive residue modulo $p$ is $O( p^\eps )$.
\end{theorem}

We prove this theorem in Section \ref{primsec}.

Our proof of Theorem \ref{eh1} does not easily allow one to convert partial progress on the Elliott-Halberstam conjecture to partial progress on Vinogradov's conjecture.  We now present a different argument that replaces the Elliott-Halberstam conjecture by a conjecture on ``Type II sums'' of the type introduced\footnote{Zhang also considered ``Type I'' and ``Type III'' sums, which will not be of direct relevance in this paper, although the $\tau_3$ distribution estimates mentioned in Section \ref{variant-sec} are related to the Type III sums of Zhang.  Similar sums had also been previously considered by Bombieri, Fouvry, Friedlander, and Iwaniec \cite{bfi, bfi-2, bfi-3, fouvry, ft, fi, fi-2, fik-3}.} by Zhang \cite{zhang}, with the feature that partial progress on the Type II conjecture implies partial progress on Vinogradov's conjecture.  In particular, the Type II estimates in \cite{polymath8a} can be used to improve slightly upon the Vinogradov bound \eqref{vino} by a method different than the Burgess argument, although the numerical exponent obtained is inferior to that in \cite{burgess}.

Let us first state the Type II conjecture, in a formulation suited for the current application.

\begin{conjecture}[Type II conjecture]\label{geh-conj} Let $0 < \varpi < 1/4$, and let $\delta > 0$ be a sufficiently small fixed quantity depending on $\vartheta$.  Let $x$ be an asymptotic parameter going to infinity.  Let $P$ be any number which is the product of some subset of the primes in $[1,x^\delta]$; equivalently, let $P$ be a square-free number all of whose prime factors are at most $x^\delta$. Let $N,M$ be quantities such that
$$x^{1/2-2\varpi} \ll N \ll M \ll x^{1/2+2\varpi}$$
with $NM \sim x$, and let $\alpha, \beta \colon \N \to \R$ be sequences supported on $[M,2M]$ and $[N,2N]$ respectively, such that one has the pointwise bounds
\begin{equation}\label{ab-div}
 |\alpha(n)| \ll 1
\end{equation}
for all natural numbers $n$.  We also assume that $\beta$ is simply the indicator function
$$ \beta = 1_{[N,2N]}.$$
Then one has
\begin{equation}\label{qq-gen}
\sup_{1 \leq a \leq x: (a,P)=1} \sum_{r \ll x^{1/2+2\varpi}: r|P} |\Delta(\alpha \star \beta; a\ (r))| \ll x \log^{-A} x
\end{equation} 
for any fixed $A > 0$.
\end{conjecture}

This conjecture is implied by the generalised Elliott-Halberstam conjecture in \cite{polymath8b}, which was in turn inspired by a similar conjecture in \cite{bfi}.  In \cite{motohashi} (see also \cite{gallagher}), a generalisation of the Bombieri-Vinogradov theorem is obtained which roughly speaking implies (up to logarithmic factors) the $\varpi=0$ endpoint of this conjecture.  The arguments in \cite{zhang} implicitly establish the above conjecture for $0 < \varpi < \frac{1}{1168}$, and more explicitly the estimate in \cite[Theorem 5.1(iv)]{polymath8a} establishes the conjecture for $0 < \varpi < \frac{1}{68}$.  The estimates in those papers allow for more general values of $a,r$ and more general sequences $\alpha$, $\beta$ than those considered here; however, the restricted version of Conjecture \ref{geh-conj} stated above will suffice for our application.  It is likely that the additional restrictions imposed here (particularly the requirement that $\beta$ be the indicator function of an interval) allow for some improvement in the exponent $\frac{1}{68}$ obtained in \cite{polymath8a}; see also Section \ref{variant-sec} below for a slightly different way to improve upon this exponent, from $\frac{1}{68}$ to $\frac{1}{28}$.

Our next main result is then 

\begin{theorem}[Type II sums bound character sums]\label{eh3}  Suppose that Conjecture \ref{geh-conj} holds for a fixed choice of $0 < \varpi < \frac{1}{4}$.  Then one has
\begin{equation}\label{slit}
\left|\sum_{n < q^{1/2 - 2 \varpi + \eps}} \chi(n)\right| \ll q^{1/2 - 2 \varpi + \eps} \log^{-A} q
\end{equation}
for any sufficiently small fixed $\eps > 0$, any fixed $A>0$, and any natural number $q$ (not necessarily prime), whenever $\chi$ is a non-principal primitive Dirichlet character of conductor $q$.
\end{theorem}

By the usual argument of Vinogradov this gives

\begin{corollary}\label{eh3-cor} Suppose that Conjecture \ref{geh-conj} holds for a fixed choice of $0 < \varpi < \frac{1}{4}$.  Then one has
$$
n(q) \ll q^{\frac{1}{\sqrt{e}} (\frac{1}{2} - 2 \varpi) + \eps} $$
for any fixed $\eps > 0$ and any prime $q$.
\end{corollary}

\begin{proof}  From the pointwise estimate
$$ \chi(n) \geq 1 - 2 \sum_{p|n: p > n(q)} 1 $$
for the quadratic character $\chi(n) \coloneqq \left( \frac{n}{q} \right)$
we see that
$$ \sum_{n < x} \chi(n) \geq x - 1 - 2 \sum_{n(q) < p \leq x} \left(\frac{x}{p} + 1\right)$$
for any $x > 1$.  Setting $x := q^{1/2-2\varpi+\eps}$ for some $\eps>0$ and using Theorem \ref{eh3}, we see that
$$ x - 2 x \sum_{n(q) < p \leq x} \frac{1}{p} \leq o(x)$$
as $q \to \infty$.  From Mertens' theorem, this implies that
$$ \log \frac{\log x}{\log n(q)} \geq \frac{1}{2} + o(1),$$
and the claim follows.
\end{proof}

In particular, the Type II estimates in \cite{polymath8a} give the improvement 
$$
n(p) \ll p^{\frac{1}{\sqrt{e}} ( \frac{1}{2} - \frac{1}{34} ) + \eps } 
$$
to \eqref{vino} for any fixed $\eps > 0$.  This is well short of the improvement in \eqref{burgess}, however it represents a slightly different way to break the ``square root barrier'' than the Burgess argument; for instance, the arguments can extend to general moduli than primes $p$ without much difficulty, whereas the Burgess argument encounters some additional technical issues when the modulus is not cube-free.  One will be able to surpass the Burgess bound as soon as one can establish a Type II estimate for some $\varpi > \frac{1}{8}$ (or $\varpi > \frac{1}{12}$ in the non-cube-free case), thus one needs to improve the Type II exponents in \cite{polymath8a} by a factor of roughly eight.  Interestingly, it was noted in \cite{bfi} (see Conjecture 3 of that paper) that if one assumed square root cancellation in certain exponential sums, one could obtain Type II estimates for all $\varpi < \frac{1}{8}$, thus falling barely short of being able to improve upon the Burgess bound.

Theorem \ref{eh3}, when combined with the Type II estimates in \cite{polymath8a}, establishes the short character sum bounds
\begin{equation}\label{qeea}
\sum_{n < q^{\frac{1}{2} - \frac{1}{34} + \eps}} \chi(n) = q^{\frac{1}{2} - \frac{1}{34} + \eps} \log^{-A} q
\end{equation}
for any primitive character $\chi$ of conductor $q$.  This bound is inferior to that of Burgess \cite{burgess, burgess-2, burgess-3}, which establishes
$$
\sum_{M \leq n \leq M+N} \chi(n) = N^{1-\delta(\eps)}
$$
for arbitrary $M$ when $N \gg q^{1/3 + \eps}$ (if $q$ is not cube-free) or $N \gg q^{1/4+\eps}$ (if $q$ is cube-free), and $\delta(\eps)>0$ depends only on $\eps$.  With our methods, one would need Type II estimates at level of distribution at least $2/3$ (thus $\varpi > 1/12$) to improve upon the Burgess bound in the non-cube-free setting, or at least $3/4$ (thus $\varpi > 1/8$) in the cube-free setting.  Note also the Burgess bound has also been improved for certain types of modulus $q$, such as smooth numbers (see e.g. \cite{graham}, \cite{gold}) or prime powers (see e.g. \cite{postnikov}).

\begin{remark}
If one had the Type II estimates for all $0 < \varpi < 1/4$, then (by combining Corollary \ref{eh3-cor} with the Burgess bound) we would have
$$ \sum_{n \leq x} \chi(n) \ll x \log^{-A} x $$
for all $x \geq q^\eps$ and fixed $A,\eps>0$, and hence (by summation by parts) one would obtain a very slight improvement $L(1,\chi) = o(\log q)$ to the standard upper bound $L(1,\chi) = O(\log q)$ for the sum $L(1,\chi) = \sum_n \frac{\chi(n)}{n}$.  Furthermore, one obtains the bound $L(s,\chi) = O(\log^2 q)$ (say) when $|s-1| \leq \frac{A \log\log q}{\log q}$ for any fixed $A$.  Using this and standard arguments (see e.g. \cite[Chapter 8]{ik}), one can enlarge\footnote{We thank James Maynard for this remark.} the classical zero-free region of $L(s,\chi)$ to include the region $|s-1| \leq \frac{A}{\log q}$ for any fixed $A>0$, except possibly for a Siegel zero.  This in turn can be used to improve the prime number theorem of Gallagher \cite{gallagher-large}, and hence also the constant in Linnik's theorem on primes in an arithmetic progression, assuming the Type II estimates, and possibly excluding an exceptional modulus; we omit the details.
\end{remark}

\begin{remark} By standard arguments (see e.g. \cite[Corollary 9.20]{mv}) starting from the observation that the sum
$$ \sum_{d|Q} \frac{\phi(Q/d) \mu(d)}{Q} \sum_{\substack{\chi\ (Q) \\ \operatorname{ord}(\chi)=d}} \sum_{n \leq x} \chi(n)$$
counts the number of primitive roots modulo a prime $p$ up to $x$, where $Q$ is the product of all the primes dividing $p-1$, we see that Theorem \ref{eh3} implies that if one has Type II estimates for a given $0 < \varpi < 1/4$, then the least primitive root of $\Z/p\Z$ is $O( p^{1/2 - 2\varpi + \eps} )$ for any fixed $\eps$ and any prime $p$, provided that $p-1$ has at most $O( \log\log p )$ prime factors; we leave the details to the interested reader.   In particular, we can strengthen the conclusion of Theorem \ref{eh-a} slightly if we replace the Elliott-Halberstam conjecture by the Type II conjecture for $\varpi$ arbitrarily close to $1/4$.  It may be possible\footnote{We thank the anonymous referee for this suggestion.} to remove the requirement on the number of prime factors of $p-1$, by using zero-density estimates (together with a result of Rodosskii \cite{rod} linking $L$-function zeroes with character sums; see also the recent preprints \cite{banks}, \cite{gs-new}) to show that $\sum_{n \leq x} \chi(n)$ is small for most characters $\chi$; we will not pursue this in detail here.  
\end{remark}

\begin{remark} Suppose Conjecture \ref{geh-conj} holds for some fixed $0 < \varpi < 1/4$, and suppose that $q$ is a large prime such that the least prime quadratic \emph{residue} is at least\footnote{We thank John Friedlander for suggesting this problem.} $q^{1/2-2\varpi+\eps}$.  Then, letting $\chi$ be the quadratic character of conductor $q$, one has $\chi(n) = \lambda(n)$ for all $n \leq q^{1/2-2\varpi+\eps}$, where $\lambda$ is the Liouville function.  From the prime number theorem (for $n \leq q^{1/2-2\varpi+\eps}$) and Theorem \ref{eh3}, we conclude that $\sum_n \frac{\chi(n)}{n} \ll \log^{-A} q$ and $\sum_n \frac{\chi(n) \log n}{n} \gg 1$, so that $\left| \frac{L'(1,\chi)}{L(1,\chi)}\right| \gg \log^A q$ for any fixed $A$.  From standard arguments this implies that one has a Siegel zero $L(\sigma,\chi)=0$ with $1-\sigma \ll \log^{-A} q$ for any fixed $A$.  Thus, if one could rule out Siegel zeroes, one could use Type II estimates to bound the least prime quadratic residue.  If one could improve the $\log^{-A} q$ gain in \eqref{slit} to a power saving $q^{-\eps}$, then Siegel's theorem could be used to remove the need to consider Siegel zeroes; for instance this argument recovers the standard bound of $q^{1/4+o(1)}$ for the least prime quadratic residue coming from the Burgess bound.  However, our arguments would require a similar power saving in the Type II estimates to achieve this, which may be an overly ambitious hypothesis.
\end{remark}

We prove Theorem \ref{eh3} in Section \ref{geh-sec}.  The idea here is to exploit the fact that if $\sum_{n \in [N/2,N]} \chi(n)$ is large, then on an interval $[1, x]$ with $x = q^{1+O(\eps)}$, $\chi(n)$ will exhibit large correlation with $\alpha * \beta(n+jq)$ for any $j = O( q^\eps)$, where $\beta \coloneqq 1_{[N/2,N]}$ and $\alpha$ is the restriction of $\chi$ to smooth squarefree numbers of magnitude close to $x/N$ and which are coprime to $q$.  This is because of the multiplicativity and periodicity properties of $\chi$.  An application of Cauchy-Schwarz (i.e. the dispersion method) then shows that $\alpha * \beta(n+jq)$ and $\alpha * \beta(n + j'q)$ correlate with each other for some distinct $j,j'$, but one can use Type II estimates to exclude this scenario from occurring.

\begin{remark} The above argument shares many similarities with the argument of Burgess \cite{burgess}.  Both arguments rely heavily on the periodicity and multiplicativity of the Dirichlet character $\chi$, which allows one to start with a hypothesis that a single character sum $\sum_{n \leq x} \chi(n)$ is large, and deduce that $\chi$ is biased on many arithmetic progressions.  In the current argument, one exploits the bias of $\chi$ on medium-length arithmetic progressions (of length about $q^{1/2 - 2\varpi}$) and varying modulus; in contrast, the argument of Burgess exploits the bias of $\chi$ on many (close to $q^{1/2}$) very short progressions (of length $q^\eps$ for some small $\eps$) and fixed modulus.  Unfortunately, the author was not able to combine the two methods together to obtain any improvement on \eqref{burgess}, without assuming a large portion of the Elliott-Halberstam or Type II conjectures.
\end{remark}

\begin{remark}  The proof of Theorem \ref{eh3} may possibly extend to cover the shifted character sums $\sum_{M \leq n \leq M+N} \chi(n)$ appearing in the work of Burgess; however, the way the argument is currently presented, this would require a shifted version of a Type II estimate in which the convolution $\alpha * \beta$ is replaced by a shifted convolution.  As such, one can no longer directly quote the results from \cite{polymath8a} to obtain a result for such shifted sums; however it is plausible that some modification of the \emph{proof} of the Type II estimate in \cite{polymath8a} can still be adapted to this shifted setting.  We do not pursue this matter here (as with the centred sums, the we do not seem to directly improve upon the Burgess bounds at the current level of technology for equidistribution estimates).
\end{remark}

A variant of the argument used to prove of Theorem \ref{eh3}, which we discuss in Section \ref{variant-sec} below, allows one to use distributional estimates for the higher divisor functions
\begin{equation}\label{div}
\tau_k(n) := \sum_{n_1,\dots,n_k: n_1 \dots n_k = n} 1
\end{equation}
(or more precisely, from dyadic components of such functions) in place of Type II estimates to obtain similar results.  Roughly speaking, a distributional estimate on $\tau_k$ at level $\theta$ implies a bound of the form \eqref{slit} with $\frac{1}{2}-2\varpi$ replaced by $\max( 1-\theta, \frac{1}{k\theta+1} )$; thus for instance the classical distribution estimate of $\tau_2$ at $\theta=\frac{2}{3}$ gives \eqref{slit} with $\varpi = \frac{1}{28}$, slightly improving upon \eqref{qeea}, though still short of the Burgess bounds in both cube-free and non-cubefree cases.  More recently, a level of distribution $4/7$ has been established (in a restricted averaged sense) for $\tau_3$ in \cite{FKI}, which (morally at least) also recovers \eqref{slit} with $\varpi = \frac{1}{28}$.  To improve upon the Burgess bound, one would need $\tau_k$ at level of distribution above $2/3$ for some $k \geq 3$ (in the non-cube-free case) or above $3/4$ for some $k \geq 4$ (in the cube-free case).  Both results seem unfortunately to be out of reach of current methods.  

A similar analysis, again discussed in Section \ref{variant-sec} below suggests that one should be able to improve the exponent $\frac{1}{2}-2\varpi$ in \eqref{slit} to $\frac{1}{k}-c$ for some $c>0$ provided that one can obtain good asymptotics for sums such as
$$ \sum_{n \leq x} \tau_k(n) \tau_k(n+q)$$
with $q = o(x)$.  In particular, controlling such sums for $k=3$ would (morally, at least) improve upon the non-cube-free Burgess bound, and for $k=4$ would improve upon the cube-free Burgess bound.  Unfortunately, rigorous asymptotics for these sums have only been established for $k=2$.

\subsection{Notation}\label{notation}

We use the following asymptotic notation.  We allow for an asymptotic parameter (e.g. $x$ or $q$) to go to infinity; quantities in this paper may depend on this parameter unless they are explicitly labeled as \emph{fixed}.  We then write $X \ll Y$, $X = O(Y)$, or $Y \gg X$ if one has $|X| \leq CY$ for some fixed $C$ (in particular, $C$ can depend on other parameters as long as they are also fixed).  We also write $X = o(Y)$ if we have $|X| \leq c Y$ for some quantity $c$ that goes to zero as the asymptotic parameter goes to infinity, and write $X \sim Y$ for $X \ll Y \ll X$.

Sums over $p$ are understood to be over primes, and all other sums are over the natural numbers $\N = \{1,2,3,\dots\}$ unless otherwise indicated.

Given two functions $f, g \colon \N \to \C$, their Dirichlet convolution $f*g$ is defined by
$$ f*g(n) \coloneqq \sum_{d|n} f(d) g(\frac{n}{d}),$$
where $d|n$ denotes the assertion that $d$ divides $n$.

Given two natural numbers $a,b$, we use $(a,b)$ to denote the greatest common divisor of $a,b$, and $a\ (b)$ to denote the residue class of integers equal to $a$ modulo $b$. Given a natural number $r$, we use $(\Z/r\Z)^\times = \{ a\ (r): (a,r)=1\}$ to denote the primitive residue classes modulo $r$.

We use $1_E$ to denote the indicator function of $E$, thus $1_E(n)$ equals $1$ when $n \in E$ and equals zero otherwise.  Similarly, if $S$ is a sentence, we write $1_S$ to equal $1$ when $S$ is true and $0$ otherwise, thus for instance $1_E(n) = 1_{n \in E}$.

\subsection{Acknowledgments}

The author was supported by a Simons Investigator grant, the
James and Carol Collins Chair, the Mathematical Analysis \&
Application Research Fund Endowment, and by NSF grant DMS-1266164.  He also thanks John Friedlander, Andrew Granville, James Maynard, Lillian Pierce, and Felipe Voloch for several useful discussions, and the anonymous referee for many valuable comments and suggestions.

\section{Vinogradov from Elliott-Halberstam}\label{eh-sec}

We now prove Theorem \ref{eh1}.  We will in fact prove a slightly stronger implication, in which Conjecture \ref{vinconj} is replaced by 

\begin{conjecture}\label{vinconj-2}  For any Dirichlet character $\chi$, let $n_\chi$ be the first natural number with $\chi(n_\chi) \neq 1$.  For any fixed $\eps>0$, we have $n_\chi \ll q^\eps$ for any primitive Dirichlet character $\chi$ of prime conductor $q$.
\end{conjecture}

Clearly, Conjecture \ref{vinconj} is the special case of Conjecture \ref{vinconj-2} in which $\chi$ is a quadratic character.

Assume the Elliott-Halberstam conjecture.  Suppose for sake of contradiction that Conjecture \ref{vinconj} failed, then we can find a fixed $\kappa > 0$ and a sequence $q$ of primes going to infinity, as well as a character $\chi$ of modulus $q$, such that
$$ n_\chi > q^\kappa.$$
Without loss of generality we may take $\kappa$ to be small, e.g., $\kappa < \frac{1}{2}$.  We view $q$ as an asymptotic parameter for the purposes of asymptotic notation, and reserve the right to refine $q$ to subsequences as necessary.

We will need some basic results from the theory of mean values of multiplicative functions in order to produce some anomalous distribution for $\chi(n) \Lambda(n)$ at large scales.  This could be accomplished using the results of Granville and Soundararajan \cite{gs} (or even the earlier work of Wirsing \cite{wirsing}), but we do not need the full strength of their theory here, since we will be satisfied with an analysis of logarithmic densities such as $\frac{1}{\log x} \sum_{n \leq x} \frac{\chi(n)}{n}$ instead of natural densities such as $\frac{1}{x} \sum_{n \leq x} \chi(n)$.  As such, we give a self-contained treatment here.

It will be technically convenient to work in the asymptotic limit in which we extract the mean value after sending $q$ to infinity (this is a luxury available in the logarithmic density setting that is not easily achievable for natural densities, at least if one is not willing to use the tools of nonstandard analysis).  
For any fixed $t \geq 0$, we consider the logarithmic densities
$$ A_q(t) \coloneqq \frac{1}{\log q} \sum_{n < q^t} \frac{\chi(n)}{n} $$
and
$$ B_q(t) \coloneqq \frac{1}{\log q} \sum_{n < q^t} \frac{\chi(n) \Lambda(n)}{n}.$$
From Mertens' theorem we have the Lipschitz bounds
\begin{equation}\label{qqfg}
 |A_q(t) - A_q(s)|, |B_q(t)-B_q(s)| \leq |t-s| + o(1)
\end{equation}
for all fixed $t,s \geq 0$; also we clearly have $A_q(0)=B_q(0) = 0$.  From the Arzela-Ascoli theorem, and refining $q$ to a subsequence as necessary, we may thus find \emph{fixed} Lipschitz functions $A, B \colon [0,+\infty) \to \C$ such that
\begin{equation}\label{flimit}
 A_q(t) = A(t)+o(1); \quad B_q(t) = B(t)+o(1)
\end{equation}
for all fixed $t \geq 0$.  From \eqref{qqfg} we have
$$ |A(t)-A(s)|, |B(t)-B(s)| \leq |t-s|$$
for all fixed $t,s \geq 0$.  By the Rademacher differentiation theorem, we can thus find Lebesgue measurable functions $a, b \colon [0,+\infty) \to \C$ bounded in magnitude by $1$, defined up to almost everywhere equivalence, such that
$$ A(t) = \int_0^t a(u)\ du; \quad B(t) = \int_0^t b(u)\ du$$
for all $t \in [0,+\infty)$.

We now establish some bounds on $A, B$.  Since $\chi$ has mean zero on intervals of length $q$, it is easy to see that
$$ A_q(t) = A_q(t') + o(1)$$
for all fixed $t,t' > 1$; in fact one can extend this to $t,t'>1/4$ using the Burgess bound \cite{burgess}, but we will not need to do so here.  This implies that $a$ is supported on $[0,1]$ (modulo null sets).

Next, since $\chi(n)=1$ for $n \leq q^\kappa$, we have from Mertens' theorem that
$$ A_q(t), B_q(t) = t + o(1)$$
for $t < \kappa$.  Thus $A(t)=B(t)=t$ for $t < \kappa$, and so $a(t)=b(t)=1$ for $t < \kappa$ (again up to null sets).

Next, we claim that $a,b$ obey the integral equation of Wirsing \cite{wirsing}:

\begin{lemma}[Wirsing equation]\label{wirsing}  We have
$$ t\,a(t) = \int_0^t a(u) b(t-u)\ du$$
for almost all $t > 0$.
\end{lemma}

This equation also holds for other means than logarithmic densities (replacing $a$, $b$ by suitable substitutes, such as the functions $t \mapsto \frac{1}{q^t} \sum_{n \leq q^t} \chi(n)$ and $t \mapsto \frac{1}{q^t} \sum_{n \leq q^t} \chi(n) \Lambda(n)$ respectively), but the arguments are more complicated, and one has to work non-asymptotically and admit some $o(1)$ errors; see \cite{wirsing}, \cite{gs}.

\begin{proof}  We start with the Dirichlet convolution identity
$$ \chi(n) \log n = (\chi \Lambda) * \chi(n)$$
and conclude for any fixed $t>0$ that
\begin{equation}\label{toast}
 \frac{1}{\log^2 q} \sum_{n \leq q^t} \frac{\chi(n) \log n}{n} = \frac{1}{\log q} \sum_{d \leq q^t} \frac{\chi(d) \Lambda(d)}{d} \frac{1}{\log q} \sum_{m \leq q^t/d} \frac{\chi(m)}{m}.
\end{equation}
To estimate this expression we use a Riemann sum argument.  Let $J>0$ be a large fixed natural number.  If $q^{(j-1)t/J} \leq d < q^{jt/J}$ for some $1 \leq j \leq J$, then $\frac{1}{\log q} \sum_{m \leq q^t/d} \frac{\chi(m)}{m} = A(t - \frac{jt}{J}) + O(\frac{1}{J}) + o(1)$ (with implied constant uniform in $J$), and so the expression \eqref{toast} may be written (after using Mertens' theorem to estimate error terms) as
$$
\left( \sum_{j=1}^J A(t - \frac{jt}{J}) \frac{1}{\log q} \sum_{q^{(j-1)t/J} \leq d < q^{jt/J}} \frac{\chi(d) \Lambda(d)}{d} \right) + O\left(\frac{1}{J}\right) + o(1).$$
One has 
\begin{align*}
\frac{1}{\log q} \sum_{q^{(j-1)t/J} \leq d < q^{jt/J}} \frac{\chi(d) \Lambda(d)}{d} &= B(jt/J) - B((j-1)t/J) + o(1) \\
&= \int_{(j-1)t/J}^{jt/J} b(u)\ du + o(1) 
\end{align*}
and so (by the Lipschitz nature of $A$), the previous expression becomes
$$ \int_0^1 A(t-u) b(u)\ du + O\left(\frac{1}{J}\right) + o(1).$$
As $J$ can be arbitrarily large, we conclude that
$$  \frac{1}{\log^2 q} \sum_{n \leq q^t} \frac{\chi(n) \log n}{n}  = \int_0^t A(t-u) b(u)\ du  + o(1).$$
On the other hand, from the identity $\frac{\log n}{\log q} = t - \int_0^t 1_{n \leq q^u}\ du$ and \eqref{flimit} we see (after a Riemann sum argument as before) that
$$ \frac{1}{\log^2 q} \sum_{n \leq q^t} \frac{\chi(n) \log n}{n} = t A(t) - \int_0^t A(u)\ du + o(1)$$
and hence
$$ t\,A(t) - \int_0^t A(u)\ du = \int_0^t A(t-u) b(u)\ du$$
for all $t$.  Differentiating using the Lebesgue differentiation theorem, we conclude that
$$ t\,a(t) = \int_0^t a(t-u) b(u)\ du $$
almost everywhere, as desired.  
\end{proof}

We will use this equation, together with some complex analysis and the previously established compact support of $a$, to derive the following consequence:

\begin{corollary} $b$ is not compactly supported (up to null sets).  
\end{corollary}

\begin{proof}  Suppose for contradiction that $b$ is compactly supported (modulo null sets).  Now consider the Fourier-Laplace transforms
$$ {\mathcal L} a(s) \coloneqq \int_0^\infty a(t) e^{-ts}\ dt$$
and
$$ {\mathcal L} b(s) \coloneqq \int_0^\infty b(t) e^{-ts}\ dt;$$
as $a$ and $b$ are both bounded and compactly supported, the functions ${\mathcal L} a, {\mathcal L} b$ are entire and of at most exponential growth, and are not identically zero since $a,b$ are not identically zero.  On the other hand, from Lemma \ref{wirsing} and standard computations we have
\begin{equation}\label{dsl}
 - \frac{d}{ds} {\mathcal L} a = {\mathcal L} a \times {\mathcal L} b.
\end{equation}
As ${\mathcal L} b$ has no poles, ${\mathcal L} a$ cannot have any zeroes; in particular, $\log {\mathcal L} a$ is entire and at most linear growth, and must therefore be a linear function, so that ${\mathcal L} a$ is an exponential function, and hence by \eqref{dsl} ${\mathcal L} b$ is a constant function.  But this is absurd (it contradicts the Riemann-Lebesgue lemma).
\end{proof}

\begin{remark}  The above argument shows that $a$ and $b$ cannot both be compactly supported while still obeying Lemma \ref{wirsing}, except in trivial cases.  A stronger result in this regard, in which $a,b$ are allowed to decay exponentially, can be found in \cite{gs0}.  Note that the argument used to establish this corollary would have been significantly messier if one had to contend with $o(1)$ errors in the Wirsing integral equation, as one would need quantitative approximate versions of various basic qualitative facts about entire functions.  This is the main reason why we took the asymptotic limit $q \to \infty$ previously.  However, Andrew Granville (private communication) has informed me that such an approximate version of this observation was obtained in an unpublished work of Granville and Soundararajan.  (See also the recent paper \cite{gs-new} for some related results.)
\end{remark}

From the above corollary and the Lebesgue differentiation theorem, we can find fixed $1 < t_1 < t_2$ such that $|B(t_2) - B(t_1)| > 0$, and so
$$
\left|\frac{1}{\log q} \sum_{q^{t_1} < n < q^{t_2}} \frac{\chi(n) \Lambda(n)}{n}\right| \gg 1$$
for $q$ sufficiently large.  By the pigeonhole principle, we may thus find $q^{t_1} \ll x \ll q^{t_2}$ such that
$$ |\sum_{n\in [x/2,x]} \chi(n) \Lambda(n)| \gg x.$$
Of course, $x$ will depend on $q$.  Since $q = o(x)$, we may shift $n$ by $q$, using the periodicity of $\chi$, to conclude that
$$ \left|\sum_{n\in [x/2,x]} \chi(n) \Lambda(n+q)\right| \gg x.$$
On the other hand, as $\chi$ has mean zero on intervals of length $q$, we have
$$ \sum_{n\in [x/2,x]} \chi(n)  = o(x).$$
Thus if we let
$$ X \coloneqq \sum_{n\in [x/2,x]} \chi(n) (\Lambda(n+q)-1)$$
then we have
\begin{equation}\label{Xx}
 |X| \gg x
\end{equation}
for sufficiently large $q$.

We now upper bound $X$ in order to contradict \eqref{Xx}.  The first step is to expand out $\chi$ in terms of Dirichlet convolutions.
By M\"obius inversion, we can express 
$$ \chi = 1 * f = 1 + 1 * \tilde f$$
where
$$ \tilde f(n) \coloneqq f(n) - 1_{n=1}$$
and
$$ f = \chi * \mu;$$
in other words, $f$ is the multiplicative function with
$$ f(p^j) = \chi(p)^{j-1} (\chi(p)-1)$$
whenever $p$ is a prime and $j \geq 1$, with the convention that $0^0=1$.  In particular we see that $f(n)$ is only non-zero when $n$ is \emph{$q^\kappa$-rough}, by which we mean that $n$ has no prime factor less than or equal to $q^\kappa$; this implies furthermore that $\tilde f(n)$ vanishes unless $n > q^\kappa$, and that
\begin{equation}\label{fbound}
|\tilde f(n)| \ll 1
\end{equation}
whenever $n = O( q^{O(1)} )$.

Let $\nu > 0$ be a small fixed constant to be chosen later.  
We expand $X$ using the identity
\begin{equation}\label{foo}
 \chi 1_{[x/2,x]} = 1_{[x/2,x]} + (1_{[1,x^\nu)} * \tilde f) 1_{[x/2,x]} + (1_{[x^\nu,q^{-\kappa}x]} * \tilde f) 1_{[x/2,x]}
\end{equation}
where we have used the fact that $\tilde f(n)$ vanishes for $n < q^\kappa$.   This gives the splitting
$$ X = X_1 + X_2 + X_3$$
where
\begin{align*}
X_1 &= \sum_{n\in [x/2,x]} (\Lambda(n+q)-1) \\
X_2 &= \sum_{n\in [x/2,x]} (1_{[1,x^\nu)} * \tilde f)(n) (\Lambda(n+q)-1) \\
X_3 &= \sum_{n\in [x/2,x]} (1_{[x^\nu, q^{-\kappa} x]} * \tilde f)(n) (\Lambda(n+q)-1).
\end{align*}
From the prime number theorem we have
$$ X_1 = o(x).$$
For $X_2$, we use the triangle inequality to bound
$$ |X_2| \leq \sum_{d < x^\nu} \sum_{\frac{x}{2d} \leq m \leq \frac{x}{d}} |\tilde f(m)| (\Lambda(dm+q)+1)$$
We claim that
\begin{equation}\label{ddo}
 \sum_{\frac{x}{2d} \leq m \leq \frac{x}{d}} |\tilde f(m)| \Lambda(dm+q) \ll \frac{x}{\phi(d) \log x} 
\end{equation}
and
\begin{equation}\label{ddo-2}
 \sum_{\frac{x}{2d} \leq m \leq \frac{x}{d}} |\tilde f(m)| \ll \frac{x}{d \log x} 
\end{equation}
for all $d < x^\nu$, and hence
$$ X_2 \ll \nu x$$
with implied constant independent of $\nu$.

We first prove \eqref{ddo}.  From \eqref{fbound} we have $|\tilde f(m)| \Lambda(dm+q) = O(\log x)$, and this expression vanishes unless $m$ and $dm+q$ are both $q^\kappa$-rough, except for a small exceptional contribution (coming from when $dm+q$ is the power of a small prime) that can easily be seen to be negligible.  Removing this exceptional contribution, we see that we are removing two residue classes mod $p$ from the interval of $m$ for each prime $p < x^\kappa$ not dividing $d$.  Using a standard upper bound sieve (see e.g. \cite{FRI}), we conclude that the number of surviving summands $m$ is $O( \frac{x}{\phi(d) \log^2 x} )$, and the claim follows.  The bound \eqref{ddo-2} is established similarly, except now we bound $|\tilde f(m)| = O(1)$ and we remove just a single residue class for each prime $p$, rather than two.

Finally we turn to $X_3$.  We expand
$$
X_3 = \sum_{q^\kappa \ll r \ll x^{1-\nu}} \tilde f(r) \sum_{m \in [\frac{x}{2r},\frac{x}{r}] \cap [x^\nu, q^{-\kappa} x]} (\Lambda(rm+q)-1).$$
The contribution when $r \sim q^\kappa$ or $r \sim x^{1-\nu}$ can be seen to be $O( \frac{x}{\log x} )$ using the Brun-Titchmarsh inequality (and upper bound sieve bounds on $q^\kappa$-rough numbers, as in the estimation of $X_2$).  The contribution when $r$ is divisible by $q$ can be treated similarly (in fact one has the better bound of $O(x/q)$ in this case).  So we may write
$$ X_3 = \sum_{2 q^\kappa < r < \frac{1}{2} x^{1-\nu}; (r,q)=1} \tilde f(r) \sum_{\frac{x}{2r} \leq m \leq \frac{x}{r}} (\Lambda(rm+q)-1) + o( x )$$
or equivalently (since $q$ is significantly smaller than $x$)
$$ X_3 = \sum_{2 q^\kappa < r < \frac{1}{2} x^{1-\nu}; (r,q)=1} \tilde f(r) \sum_{n \in [x/2,x]: n = q\ (r)} (\Lambda(n)-1) + o( x ).$$
Invoking the Elliott-Halberstam conjecture and the prime number theorem, we then have
$$ X_3 = \sum_{2 q^\kappa < r < \frac{1}{2} x^{1-\nu}; (r,q)=1} \tilde f(r) \left(\frac{1}{\phi(r)} \frac{x}{2} - \frac{1}{r} \frac{x}{2}\right) + o( x ).$$
If $r$ contributes to the above sum, then it is the product of $O(1)$ primes of size at least $q^\kappa$, and so $\frac{1}{\phi(r)} = \frac{1}{r} + O( q^{-\kappa} \frac{1}{r} )$.  From this we see that
$$ X_3 = o(x).$$
Putting all this together, we conclude that
$$ |X| \ll (\nu+o(1)) x,$$
contradicting \eqref{Xx} for $\nu$ small enough. This completes the proof of Theorem \ref{eh1}.

\begin{remark} Our arguments here do not easily give any effective quantitative bound on $n(p)$ due to our use of asymptotic limits; in particular, the fixed quantities $t_1,t_2$ appearing above were obtained by what is essentially a compactness argument, and thus not obviously effective.  It is likely that a more carefully quantitative version of the above argument (perhaps using the estimates from \cite{gs}) can make this portion of the argument effective, thus allowing one to derive partial progress on the Vinogradov conjecture from sufficiently strong partial progress on the Elliott-Halberstam conjecture; however, the dependence of constants will be far worse than in Theorem \ref{eh3}.  We will not pursue this question further here.
\end{remark}

\begin{remark} Suppose the Burgess bound \eqref{burgess} was sharp up to epsilon factors, in the sense that one could find a sequence of primes $q$ going to infinity with $n(q) = q^{\frac{1}{4\sqrt{e}}+o(1)}$.  Then by extracting a limit to obtain the functions $a,b$ as above, we see that $a(t)=b(t)=1$ for $t \leq \frac{1}{4\sqrt{e}}$ and (from the Burgess character sum bounds) $a(t) = 0$ for $t > \frac{1}{4}$.  As was first observed by Heath-Brown (see e.g. Appendix 2 of \cite{diamond}), this information allows one in this case to determine the functions $a$ and $b$ completely.  Indeed, in the range $\frac{1}{4\sqrt{e}} \leq t < \frac{1}{2\sqrt{e}}$ one has from Lemma \ref{wirsing} that
$$ t\,a(t) = \int_0^t a(u)\ du - \int_0^{t-1/4\sqrt{e}} (1-b(t-u))\ du.$$
Bounding $1-b(t-u)$ by $2$, we thus have
$$ t\,a(t) \geq \int_0^t a(u)\ du - 2 (t-1/4\sqrt{e})$$
and thus by Gronwall's inequality
$$ a(t) \geq 1 - 2 \log(4 \sqrt{e} t).$$
(Indeed, one can verify that the difference $f(t) := a(t) - 1 + 2 \log(4\sqrt{e} t)$ obeys the inequality $t\ f(t) \geq \int_{1/4\sqrt{e}}^t f(u)\ du$ for $\frac{1}{4\sqrt{e}} \leq t < \frac{1}{2\sqrt{e}}$ with $f(\frac{1}{4\sqrt{e}}) = 0$.)  Since equality is attained for $t=1/4$ (note from Lemma \ref{wirsing} that $a$ is continuous), we must have $1-b(t-u)=2$ whenever $t \leq 1/4$ and $0 \leq u < t-1/4\sqrt{e}$, that is to say $b(t)=-1$ for $1/4\sqrt{e} < t \leq \frac{1}{4}$; also $a(t) = 1 - 2 \log(4\sqrt{e} t)$ in this range.  For $t>1/4$, Lemma \ref{wirsing} gives
$$ 0 = \int_0^t a(t-u) b(u)\ du$$
which on differentiation gives the integral equation
$$ b(t) = 2 \int_{1/4\sqrt{e}}^{1/4} b(t-u) \frac{du}{u}$$
which can then be used to complete the description of $b$, for instance via Laplace transforms.  For instance we see that $b(t)=1$ for $1/4 < t \leq \frac{1}{2\sqrt{e}}$.  One can compute that $b$ does not vanish near $t=1$, in which case the argument above shows that some improvement upon \eqref{burgess} can be made provided one can establish the Elliott-Halberstam conjecture for some $\vartheta > 1 - \frac{1}{4\sqrt{e}} \approx 0.8484$.
\end{remark}

\section{From Elliott-Halberstam to the least primitive root}\label{primsec}

We now prove Theorem \ref{eh-a}.  The key new tool is the following combinatorial statement.  Given a subset $A$ of an additive group $G = (G,+)$ and a natural number $k$, define the iterated sumset $kA$ to be the set of all sums $a_1+\dots+a_k$ where $a_1,\dots,a_k$ are elements in $A$ (allowing repetition).

\begin{proposition}[Escape from cosets]\label{for}  Let $d, m \geq 1$ be fixed integers.  Then there exists a natural number $k$ with the following property: whenever $G$ is a finite additive group whose order is the product of at most $d$ primes (counting multiplicity), and $A$ is a subset of $G$ containing zero for which one has inclusions of the form
$$ kA \subset \bigcup_{i=1}^m x_i + H_i \subsetneq G$$
for some cosets $x_i + H_i$ of subgroups $H_i$ of $G$, then $A$ is contained in a proper subgroup of $G$.
\end{proposition}

In the contrapositive, Proposition \ref{for} asserts that if $A$ generates $G$ and contains $0$, then the iterated sumsets $kA$ for $k$ large enough cannot be covered by a small number of cosets of subgroups of $G$, unless these cosets of subgroups already covered all of $G$.  Thus the sumsets $kA$ ``escape'' all non-trivial unions of boundedly many cosets. This result can be viewed as a simple abelian variant of the nonabelian ``escape from subvarieties'' lemma that first appeared in \cite{emo}.

Let us assume this proposition for the moment and see how it implies Theorem \ref{eh-a}.  Assume the Elliott-Halberstam conjecture, and assume for sake of contradiction that the conclusion of Theorem \ref{eh-a} failed.  Carefully negating the quantifiers, this means that we can find a sequence of primes $p$ going off to infinity, with $p-1$ being the product of $O(1)$ primes, and a fixed $\kappa>0$, with the property that the least primitive root of $\Z/p\Z$ is at least $p^\kappa$.

Using a discrete logarithm, we have an isomorphism $\log: (\Z/p\Z)^\times \to G$ from the multiplicative group $(\Z/p\Z)^\times$ to the additive cyclic group $G := \Z/(p-1)\Z$.  If $n$ is a natural number less than $p^\kappa$, then by hypothesis $n$ is not a primitive root of $(\Z/p\Z)^\times$, which implies that
$$ \log(n) \subset \bigcup_{r|p-1: r < p-1} \{ x \in G: rx = 0\} \subsetneq G.$$
In particular, for any natural number $k$, if we set $A := \{ \log(n): 1 \leq n < p^{\kappa/k} \}$, then
$$ kA \subset \bigcup_{r|p-1: r < p-1} \{ x \in G: rx = 0\} \subsetneq G.$$
Since $\log(1)=0$, $A$ contains $0$.  Applying Proposition \ref{for} (and using the hypothesis that $p-1$ is the product of $O(1)$ primes), we conclude (for $k$ large enough) that $A$ is contained in a proper subgroup of $G$.  Equivalently, $A$ lies in the kernel of a primitive character $\chi$ of conductor $p$, thus $\chi(n) = 1$ for all $n < p^{\kappa/k}$.  But this contradicts Conjecture \ref{vinconj-2}, which as we saw in the previous section was a consequence of the Elliott-Halberstam conjecture.

It remains to prove Proposition \ref{for}.  To illustrate the proposition, let us first give a simple case when $G$ is a direct product $H_1 \times H_2$ and we are given that $0 \in A$ and
$$ 2A \subset (H_1 \times \{0\}) \cup (\{0\} \times H_2).$$
We claim that this forces either $A \subset H_1 \times \{0\}$ or $A \subset \{0\} \times H_2$.  Indeed, if neither of these statements were true, then either there would exist $a \in A$ that was outside both $H_1 \times \{0\}$ and $\{0\} \times H_2$, or else there would exist $a_1, a_2 \in A$ with $a_1 \in H_1 \times \{0\}$, $a_2 \in \{0\} \times H_2$, and $a_1,a_2 \neq 0$.  In either case we could find an element of $2A$ ($a+0$ or $a_1+a_2$, respectively) that was outside of $(H_1 \times \{0\}) \cup (\{0\} \times H_2)$, giving the desired contradiction.  This simple special case is already sufficient to handle the case of Theorem \ref{eh-a} in which $p-1$ is the product of just two primes (that is $p-1=2q$ for some prime $q$), although in this case it turns out that the least primitive root is also the least quadratic nonresidue (for $p$ large enough, at least), so the claim in this case is already immediate from Theorem \ref{eh1}.

The general case can be obtained by a rather complicated induction on the ``complexity'' of the covering set $\bigcup_{i=1}^m x_i + H_i$, as follows.  Fix a natural number $d$.  Define a \emph{configuration} to be a tuple
\begin{equation}\label{kga}
 (k, G, A, m, (x_i+H_i)_{i=1}^m)
\end{equation}
where $k,m$ are natural numbers, $G$ is a finite additive group with $|G|$ the product of $d$ primes, $A$ is a subset of $G$ containing $0$ and not contained in any proper subgroup of $G$, and the $x_i+H_i$ are distinct cosets in $G$, such that
\begin{equation}\label{ka}
 kA \subset \bigcup_{i=1}^m x_i + H_i \subsetneq G.
\end{equation}
In particular this implies that $H_i \neq G$ for each $i$.  Our task is to show that for any configuration \eqref{kga}, that $k$ is bounded by a quantity depending only on $d$ and $m$.

Suppose for contradiction that this claim failed.  Then we can find a sequence of configurations \eqref{kga} in which $m$ stays constant, but $k$ goes to infinity.  (The other data $G, A, x_i, H_i$ in the sequence may vary arbitrarily.)

Now we define a measure of complexity of a configuration \eqref{kga}.
Given a subgroup $H$ of $G$, define the \emph{dimension} $\operatorname{dim}(H)$ of $H$ to be the quantity such that the order $|H|$ of $H$ is the product of $\operatorname{dim}(H)$ primes (counting multiplicity).  This is a natural number between $0$ and $d$, and any proper subgroup of $G$ has dimension at most $d-1$.

Given a configuration \eqref{kga}, define the \emph{complexity} of the configuration to be the tuple $(m_0,\dots,m_{d-1})$, where for each $j=0,\dots,d-1$, $m_j$ is the number of cosets $x_i+H_i$ in the configuration such that $H_i$ has dimension $j$.  Since all the $H_i$ have dimensions between $0$ and $d-1$, we see that the $m_0,\dots,m_{d-1}$ are natural numbers that sum to $m$.  In particular, if $m$ is constant, there are only finitely many possible complexities.  Thus, by passing to a subsequence if necessary, we can find a sequence of configurations \eqref{kga} whose complexity $(m_0,\dots,m_{d-1})$ stays constant, but $k$ goes to infinity.

We give the space of tuples $(m_0,\dots,m_{d-1}) \in \N^d$ the lexicographical ordering: we write $(m_0,\dots,m_{d-1}) < (n_0,\dots,n_{d-1})$ if there exists $0 \leq i \leq d-1$ such that $m_i < n_i$, and $m_j = n_j$ for $i < j \leq d-1$.  As is well known, this makes $\N^d$ a well-ordered set.

Call a tuple $(m_0,\dots,m_{d-1})$ \emph{good} if there exists a sequence of configurations \eqref{kga} with constant complexity $(m_0,\dots,m_{d-1})$, for which $k$ goes to infinity.  We have seen that there is at least one good tuple; by the well-ordering of $\N^d$, we may thus find a minimal good tuple $(m_0,\dots,m_{d-1})$.

By rounding $k$ down to an even number and then dividing by two, we may thus find a sequence of configurations
\begin{equation}\label{k2}
 (2k, G, A, m, (x_i+H_i)_{i=1}^m)
\end{equation}
of complexity $(m_0,\dots,m_{d-1})$ with $k$ going to infinity.

Let $d_*$ be the largest $j$ for which $m_{j}$ is non-zero, thus $0 \leq d_* \leq d-1$.  (note that at least one of the $m_j$ must be non-zero, otherwise the first inclusion in \eqref{ka} could not hold).  By relabeling, we may assume without loss of generality that $H_1$ has dimension $d_*$ for any configuration \eqref{k2} in the above sequence.

Consider a configuration \eqref{k2} in the above sequence, then
$$ 2kA \subset \bigcup_{i=1}^m x_i + H_i.$$
In particular, for any $y \in kA$, we have 
$$ kA \subset 2kA \cap (2kA-y) \subset \bigcup_{i=1}^m \bigcup_{j=1}^m (x_i + H_i) \cap (x_j - y + H_j).$$
Note that the set $(x_i + H_i) \cap (x_j - y + H_j)$ is either empty, or is a coset of $H_i \cap H_j$, which has dimension at most $d_*$, with equality if and only if $H_i=H_j$ has dimension $d_*$.  In particular, since all the cosets $x_j+H_j$ are assumed distinct, we see that if $H_i$ has dimension $d_*$, there is at most one set $(x_i + H_i) \cap (x_j - y + H_j)$ which is a coset of a $d_*$-dimensional subgroup.  In particular, at most $m_{d_*}$ of the $(x_i + H_i) \cap (x_j - y + H_j)$ arise as cosets of $d_*$-dimensional subgroups. 

Now suppose that we can find $y \in kA$ such that
\begin{equation}\label{yeah}
 y \not \in \bigcup_{1 \leq j \leq m: H_j = H_1} x_j - x_1 + H_1.
\end{equation}
Then we see that $x_1 + H_1 \neq x_j - y + H_j$ for any $j=1,\dots,m$.  As such, there are now at most $m_{d_*}-1$ of the $(x_i + H_i) \cap (x_j - y + H_j)$ arise as cosets of $d_*$-dimensional subgroups.  Collecting all the cosets of the form $(x_i + H_i) \cap (x_j - y + H_j)$ and eliminating duplicates, we obtain a new configuration
$$ (k, G, A, m', (x'_i+H'_i)_{i=1}^{m'})$$
which has strictly lower complexity than $(m_0,\dots,m_{d-1})$.  By the minimality of $(m_0,\dots,m_{d-1})$, this situation can only occur for finitely many of the sequence of configurations \eqref{k2}.  Thus, after discarding finitely many terms, we may assume that the situation \eqref{yeah} does not occur for any $y \in kA$; that is to say, we have
$$ kA \subset \bigcup_{1 \leq j \leq m: H_j = H_1} x_j - x_1 + H_1.$$
This gives rise to a configuration of strictly lower complexity than $(m_0,\dots,m_{d-1})$, unless $(m_0,\dots,m_{d-1}) = (0,\dots,0,m,0,\dots,0)$ (with $m$ in the $d_*$ position), and all of the $H_j$ are equal to $H_1$.  Thus, after discarding finitely many terms in the sequence, we may assume that $H_j=H_1$ for all $j$, and so
$$ kA \subset \bigcup_{j=1}^m x_j - x_1 + H_1.$$
Intersecting this with the inclusion $kA \subset \bigcup_{j=1}^m x_j + H_1$, we again obtain a configuration of lower complexity, unless the set of cosets $\{ x_j + H_1: 1 \leq j \leq m \}$ is invariant with respect to translation by $x_1$; so by discarding another finite number of terms in the sequence, we may assume that this is the case.  By permuting indices, we can then assume that $\{ x_j + H_1: 1 \leq j \leq m \}$ is invariant under translation by $x_i$ for any $1 \leq i \leq m$.  In other words, $\{ x_j + H_1: 1 \leq j \leq m \}$ is a subgroup of the quotient group $G/H_1$, so $\bigcup_{j=1}^m x_j + H_1$ is a subgroup of $G$.  But this has to be a proper subgroup by \eqref{ka}, and so $A$ is in a proper subgroup of $G$, a contradiction.

\section{Character sums from Type II sums}\label{geh-sec}

We now prove Theorem \ref{eh3}.   Suppose that Conjecture \ref{geh-conj} holds for a fixed choice of $0 < \varpi < \frac{1}{4}$.  Let $\delta > 0$ be as in Conjecture \ref{geh-conj}; we may assume that $\delta$ is small, e.g. $\delta < 1/4$.  Let $\eps > 0$ be a sufficiently small fixed quantity depending on $\delta$.  If the claim \eqref{slit} failed, then we could find a sequence of non-principal primitive characters $\chi$ with conductor $q$ going to infinity such that
$$
\left|\sum_{n < q^{1/2 - 2 \varpi + \eps}} \chi(n)\right| \gg q^{1/2 - 2 \varpi + \eps} \log^{-A} q
$$
for some fixed $A>0$.  From the pigeonhole principle we have
\begin{equation}\label{mood}
\left|\sum_{n \in [N/2,N]} \chi(n)\right| \gg N \log^{-A} q
\end{equation}
for some $N = q^{1/2 - 2\varpi + \eps} \log^{-O(A)} q$ (of course, $N$ will depend on $q$).

Set $x \coloneqq N^{\frac{1}{1/2 - 2\varpi}}$ and $M \coloneqq x/N$, thus
$$ N = x^{\frac{1}{2} - 2\varpi}; \quad M = x^{\frac{1}{2} + 2\varpi}$$
and
\begin{equation}\label{qee}
 x \geq q^{1+2\eps}.
\end{equation}
Let ${\mathcal D}$ be the set of squarefree natural numbers in $[(1-\log^{-10A-10} x)M,M]$ whose prime factors all lie in $[q^{\eps}, x^\delta]$ not dividing $q$.  Note that the number of primes dividing $q$ may be crudely bounded by $O(\log q)$ and are thus a negligible proportion of the primes in $[q^\eps,x^\delta]$.  If $\eps$ is small enough, then the prime number theorem gives the cardinality bound
\begin{equation}\label{D-small}
|{\mathcal D}| \sim M \log^{-10A-11} x.
\end{equation}
(We allow implied constants to depend on the fixed quantities $\eps,\delta,A$.)

We now set
$$ \alpha(m) \coloneqq 1_{{\mathcal D}}(m) \overline{\chi(m)}$$
and
\begin{equation}\label{beta-def}
 \beta(n) \coloneqq 1_{[N/2,N]}(n)
\end{equation}
and consider the quantity
$$ \sum_{j \leq q^{\eps}} \sum_{n \leq x} \chi(n) \alpha*\beta(n+jq).$$
Shifting $n$ by $jq$ and using the periodicity of $\chi$, we may write this as
$$ \sum_{j \leq q^{\eps}} \sum_{jq < n \leq x+jq} \chi(n) \alpha*\beta(n).$$
Since $\alpha*\beta$ is supported on $[MN/4,MN] = [x/4,x]$, this is equal (by \eqref{qee}) to
$$ \sum_{j \leq q^{\eps}} \sum_n \chi(n) \alpha * \beta(n)$$
which factorises as
$$ \sum_{j \leq q^{\eps}} \left(\sum_m \chi(m) \alpha(m)\right) \left(\sum_n \chi(n) \beta(n)\right)$$
and hence by \eqref{mood}, \eqref{D-small} we have
$$ |\sum_{n \leq x} \chi(n) \sum_{j \leq q^{\eps}} \alpha*\beta(n+jq)| \gg x q^{\eps} \log^{-11A-11} x.$$

We now ``disperse'' the $\alpha*\beta$ factors and eliminate the $\chi$ factors by a Cauchy-Schwarz argument.
Let $\gamma$ denote the quantity
\begin{equation}\label{gam}
 \gamma \coloneqq \frac{1}{x/2} \sum_n \alpha*\beta(n),
\end{equation}
which (since $\sum_n \beta(n) = (1+o(1)) \frac{N}{2})$ factorises as
\begin{equation}\label{gam2}
 \gamma = \frac{1+o(1)}{M} \sum_m \alpha(m).
\end{equation}
In particular, from \eqref{D-small} we have
\begin{equation}\label{gam-small}
\gamma = O(\log^{-10A-11} x).
\end{equation}
Since $\chi$ has mean zero on intervals of length $q$, we have
$$ |\sum_{n \leq x} \chi(n) \sum_{j \leq q^{\eps}} \gamma 1_{[x/2,x]}(n+jq)| \ll \gamma q q^{\eps} = o( x q^{\eps} \log^{-11A-11} x )$$
and thus
$$ |\sum_{n \leq x} \chi(n) \sum_{j \leq q^{\eps}} (\alpha*\beta - \gamma 1_{[x/2,x]})(n+jq)| \gg x q^{\eps} \log^{-11A-11} x.$$
Applying the Cauchy-Schwarz inequality, we conclude that
$$ \sum_{n \leq x} |\sum_{j \leq q^{\eps}} (\alpha*\beta - \gamma 1_{[x/2,x]})(n+jq)|^2 \gg x q^{2\eps} \log^{-22A-22} x,$$
which we rearrange (using the support of $\alpha*\beta - \gamma 1_{[x/2,x]}$ to remove the restriction $n \leq x$) as
\begin{equation}\label{jj'}
 |\sum_{j,j' \leq q^{\eps}} \sum_{n} (\alpha*\beta - \gamma 1_{[x/2,x]})(n) (\alpha*\beta - \gamma 1_{[x/2,x]})(n+(j'-j)q)| \gg x q^{2\eps} \log^{-22A-22} x.
\end{equation}

From the divisor bound we have $\alpha*\beta = x^{o(1)}$, and the inner sum 
$$ \sum_{n} (\alpha*\beta - \gamma 1_{[x/2,x]})(n) (\alpha*\beta - \gamma 1_{[x/2,x]})(n+(j'-j)q) $$
may then be crudely bounded as $x^{1+o(1)}$.  From this we may remove the diagonal contribution $j=j'$ from \eqref{jj'}; by symmetry
we may then reduce to the case $j' < j$.  By the pigeonhole principle, we thus have
\begin{equation}\label{jo}
|\sum_{n} (\alpha*\beta - \gamma 1_{[x/2,x]})(n) (\alpha*\beta - \gamma 1_{[x/2,x]})(n-jq)| \gg x \log^{-22A-22} x
\end{equation}
for some $1 \leq j \leq q^{\eps}$.

Let $j$ be as above.  We have
$$ \sum_n \gamma 1_{[x/2,x]}(n) \times \gamma 1_{[x/2,x]}(n-jq) = \gamma^2 \frac{x}{2} + o( x \log^{-22A-22} x ).$$
Also, the quantity $\alpha*\beta$ is supported in $[(1 - \log^{-10A-10} x) x/2, x]$. Standard divisor sum calculations using \eqref{D-small} give
\begin{equation}\label{engage}
\sum_n |\alpha*\beta(n)|  1_{[(1 - O(\log^{-10A-10} x)) x/2,x/2]}(n) = O( x \log^{-20A-21} x)
\end{equation}
and similarly
\begin{equation}\label{engage-2}
 \sum_n |\alpha*\beta(n)|  1_{[x,x(1+O( \log^{-10A-10} x))]}(n) = O( x \log^{-20A-21} x)
\end{equation}
while from \eqref{gam} one has
$$ \sum_n \alpha*\beta(n) \gamma = \gamma^2 \frac{x}{2}.$$
We conclude (using \eqref{gam-small}) that
$$ \sum_n \alpha*\beta(n) \times \gamma 1_{[x/2,x]}(n-jq) = \gamma^2 \frac{x}{2} + o( x \log^{-22A-22} x ).$$
A similar argument gives
$$ \sum_n \gamma 1_{[x/2,x]}(n) \times \alpha * \beta(n-jq) = \gamma^2 \frac{x}{2} + o( x \log^{-22A-22} x ).$$
Inserting these bounds into \eqref{jo}, we conclude that if $X$ denotes the quantity
\begin{equation}\label{Xdef}
X \coloneqq \sum_{n} \alpha*\beta(n) \alpha*\beta(n-jq)
\end{equation}
then we have
\begin{equation}\label{xpoo}
 \left|X - \gamma^2 \frac{x}{2}\right| \gg x \log^{-22A-22} x
\end{equation}
for $q$ large enough.

Now we estimate $X$ using Type II estimates, in order to contradict \eqref{xpoo}.  Expanding out the convolution $\alpha*\beta(n)$, we have
$$ X = \sum_r \alpha(r) \sum_{N/2 \leq m \leq N} \alpha*\beta(rm - jq)$$
or equivalently
$$ X = \sum_r \alpha(r) \sum_{\substack{rN/2-jq \leq n \leq rN-jq\\ n = jq\ (r)}} \alpha*\beta(n).$$
Note from the support of $\alpha$ that if $\alpha(r)$ is non-zero, then $rN/2-jq = x/2 + O( x \log^{-10A-10} x )$ and $rN-jq = x + O( x \log^{-10A-10} x )$.  A modification of \eqref{engage}, \eqref{engage-2} then shows that
$$ \sum_{\substack{rN/2+jq \leq n \leq rN+jq\\ n = jq\ (r)}} \alpha*\beta(n) = \sum_{n: n = jq\ (r)} \alpha*\beta(n) + O( \frac{x}{r} \log^{-20A-21} x )$$
and thus (by \eqref{D-small})
$$ X = \sum_r \alpha(r) \sum_{n: n = jq\ (r)} \alpha*\beta(n) + o( x \log^{-22A-22} x ).$$
From construction, we see that $jq$ is coprime to every prime between $x^{\eps}$ and $x^\delta$ that does not divide $q$, and is in particular coprime to $r$.  From the Type II estimate hypothesis, we have
$$
\sum_r |\alpha(r)| \left|\sum_{n: n = jq\ (r)} \alpha*\beta(n) - \frac{1}{\phi(r)} \sum_{n: (n,r)=1} \alpha*\beta(n)\right| \ll x \log^{-A'} x$$
for any fixed $A'>0$.  We conclude that
$$ X = \sum_r \frac{\alpha(r)}{\phi(r)} \sum_{n: (n,r)=1} \alpha*\beta(n) + o( x \log^{-22A-22} x ).$$
If $\alpha(r)$ is non-zero, then $r$ is the product of $O(1)$ primes between $q^\eps$ and $x^\delta$, and so $\frac{1}{\phi(r)} = \frac{1}{r} + O( \frac{q^{-\eps}}{r} )$; the contribution of the error $O( \frac{q^{-\eps}}{r} )$ is then $o( x \log^{-22A-22} x)$ by \eqref{gam-small}.  Also, from standard divisor bound bounds one has
$$ \sum_{n: p|n} \alpha * \beta (n) \ll \frac{x}{p} $$
for any prime $p$ between $q^\eps$ and $x^\delta$, and so
$$ \sum_{n: (n,r) \neq 1} \alpha*\beta(n) \ll q^{-\eps} x.$$
We conclude that
$$ X = \sum_r \frac{\alpha(r)}{r} \sum_{n} \alpha*\beta(n) + o( x \log^{-22A-22} x )$$
and hence by \eqref{gam}, \eqref{gam2}, \eqref{gam-small}, and the estimate $\frac{1}{r} = \frac{1}{M} + O( \frac{\log^{-10A-10} x}{M} )$ on the support of $\alpha$, one has
$$ X = \gamma^2 \frac{x}{2} + o( x \log^{-22A-22} x)$$
which contradicts \eqref{xpoo} for $x$ large enough.  This concludes the proof of Theorem \ref{eh3}.

\begin{remark} If we have $n(q) > x^\delta$, then the sequence $\alpha$ in the above argument is simply $\alpha = 1_{\mathcal D}$.  Thus, for the purposes of establishing Vinogradov's conjecture, it suffices to consider Type II sums when $\alpha$ is a sequence of the form $1_{\mathcal D}$; there is also considerable flexibility in how to choose the set ${\mathcal D}$, and other choices than the one given here are available.  For similar reasons, one can relax \eqref{qq-gen} by moving the absolute values outside of the $r$ summation.  This leads to some further numerical improvements in the $\frac{1}{68}$ exponent in \cite{polymath8a} for the purposes of the applications to Vinogradov's conjecture; see Section \ref{variant-sec} below.
\end{remark}

\section{A variant of the method}\label{variant-sec}

In this section we sketch how to modify the arguments in Section \ref{geh-sec} to be able to utilise distributional estimates for (components of) the divisor functions $\tau_k$.

We start with a similar setup with that in Section \ref{geh-sec}, namely that \eqref{mood} holds for some $N$ (and some character $\chi$ of conductor $q$ going off to infinity) and some fixed $A \geq 1$.  We set $x := q^{1+2\eps}$ for some small fixed $\eps>0$.  Let $k \geq 2$ be a fixed natural number, and suppose first that $N \leq x^{1/k}$.  Then the quantity $M := \lfloor x / N^k\rfloor$ is at least $1$.  If we set $\alpha(m) := \overline{\chi(m)} 1_{[(1- \log^{10A} x) M,M]}(m)$ and $\beta(n) := 1_{[N/2,N]}(n)$, a brief calculation similar to that in the previous section reveals that
$$ \left|\sum_{j \leq q^\eps} \sum_{n \leq x} \chi(n) \alpha * \beta^{*k}(n+jq)\right| \gg x q^\eps \log^{-(10+k)A} x$$
where $\beta^{*k}$ denotes the Dirichlet convolution of $k$ copies of $\beta$; one should think of $\beta^{*k}$ here as a component of the divisor function $\tau_k = 1^{*k}$ defined on \eqref{div}. We then approximate $\alpha * \beta^{*k}$ by $\gamma \psi( n/x)$, where
$$ \psi(t) := \int_{t_1 \dots t_k = t} 1_{[1/2,1]}(t_1) \dots 1_{[1/2,1]}(t_k) \frac{dt_1 \dots dt_{k-1}}{t_1 \dots t_k}$$
is the multiplicative convolution of $k$ copies of $1_{[1/2,1]}$, and
$$ \gamma:= \frac{1}{M (N/2)^k} \sum_n \alpha * \beta^{*k}(n).$$
A repetition of the arguments of the previous section (with $\alpha * \beta^{*(k-1)}$ playing the role of $\alpha$) then shows that there is $1 \leq j \leq q^\eps$ for which one has
$$ |X - \gamma^2 x \int_\R \psi^2(t)\ dt| \gg x \log^{-(20+2k) A} x$$
where
$$ X := \sum_n \alpha * \beta^{*k}(n) \alpha * \beta^{*k}(n-jq).$$
However, a somewhat tedious calculation (similar to that in the preceding section) shows that if one has an Elliott-Halberstam type distributional estimate for $\beta^{*k}$ on residue classes to moduli up to $MN^{k-1} \sim q^{1+2\eps}/N$, one can obtain an asymptotic of the form
$$ X = \gamma^2 x \int_\R \psi^2(t)\ dt + o(x \log^{-(20+2k) A} x)$$
giving the desired contradiction.  If $\tau_k$ has a level of distribution $\theta$ for some $0 < \theta < 1$, this suggests that we can establish cancellation in sums such as $\sum_{n \leq N} \chi(n)$ whenever $N \leq q^{1/k}$ and $q^{1+2\eps}/N \leq (N^k)^{\theta-\eps}$, which suggests that $N$ can be as low as $q^{\frac{1}{1+k\theta} +\eps}$ if $\theta > 1 - \frac{1}{k}$.  For instance, using the well-known level of distribution $\theta=2/3$ for the divisor function $\tau_2$ or for the variant $\beta*\beta$ (an old observation of Linnik and Selberg, arising from the Weil bound on Kloosterman sums), this argument gives \eqref{slit} with $\varpi = \frac{1}{28}$ (in fact one can replace $\log^{-A} q$ by a power savings, because the Linnik-Selberg argument provides such a savings in the equidistribution estimate).  Using only the elementary bound of Kloosterman \cite{kloost}, one gets a level of distribution $\theta = 4/7$, corresponding to the value $\varpi = 1/60$, thus giving a slight improvement over the P\'olya-Vinogradov bound (or even the currently best known consequence of Theorem \ref{eh3}) that requires no knowledge of the Weil conjectures.

If instead $N < q^{1/k}$, one can repeat the above analysis with the convolution $\alpha * \beta^{*k}$ replaced by $\beta_1 * \dots * \beta_k$, where $\beta_i = 1_{[N_i/2,N_i]}$ and $N_1,\dots,N_k \geq 1$ are quantities with $N = N_1 \geq N_2,\dots,N_k$ and $N_1 \dots N_k = x$.  If \eqref{mood} holds for all $N_1,\dots,N_k$, then the above analysis again leads to a contradiction if $q^{1+2\eps}/N \leq x^{\theta-\eps}$, which suggests that $N$ can be as low as $q^{1-\theta+\eps}$ if $\theta \leq 1-\frac{1}{k}$.  By a numerical coincidence, the best known distribution results (at $\theta=4/7$) on $\tau_3$, due to Fouvry, Kowalski, and Michel, correspond to the same value of $\varpi$, namely $1/28$, as the Linnik-Selberg distribution result discussed above.

In the endpoint case $N = x^{1/k}$, $\alpha$ becomes trivial and the quantity $X$ discussed above is analogous to the sum
$$ \sum_{n \leq x} \tau_k(n) \tau_k(n+jq),$$
with $jq$ being slightly smaller than $x$.  Thus, if one were able to obtain good asymptotics for such sums (with error terms which were smaller than the main term by an arbitrary power of the logarithm), one would expect to be able to obtain bounds such as \eqref{slit} with $q^{1/2-2\varpi+\eps}$ replaced by a quantity slightly smaller than $q^{1/k}$.  Unfortunately, asymptotics for such sums are currently only known for $k=2$.


\begin{thebibliography}{999}

\bibitem{ankeny}
N. C. Ankeny, \emph{The least quadratic non residue}, Ann. of Math. (2) \textbf{55}, (1952). 65--72. 

\bibitem{banks}
W. D. Banks, K. Makarov, \emph{Convolutions with probability distributions, zeros of
l-functions, and the least quadratic nonresidue}, preprint.

\bibitem{bateman}
P. Bateman, H. Diamond, Analytic number theory. 
An introductory course. Monographs in Number Theory, 1. World Scientific Publishing Co. Pte. Ltd., Hackensack, NJ, 2004. 

\bibitem{bombieri} 
E. Bombieri, \emph{On the large sieve}, Mathematika 12 (1965), 201--225.


\bibitem{bfi} E. Bombieri, J. Friedlander, H. Iwaniec, \emph{Primes in
    arithmetic progressions to large moduli}, Acta Math. 156
  (1986), no. 3--4, 203--251.

\bibitem{bfi-2} E. Bombieri, J. Friedlander, H. Iwaniec, \emph{Primes
    in arithmetic progressions to large moduli. II},
  Math. Ann. 277 (1987), no. 3, 361--393.

\bibitem{bfi-3}
E. Bombieri, J. Friedlander, H. Iwaniec, \emph{Primes in arithmetic progressions to large moduli. III}, J. Amer. Math. Soc. 2 (1989), no. 2, 215--224. 

\bibitem{burgess}
D. A. Burgess, \emph{The distribution of quadratic residues and non-residues}, Mathematika \textbf{4} (1957), 106--112. 

\bibitem{burgess-2}
D. A. Burgess, \emph{On character sums and L-series. II}, Proc. London Math. Soc. (3), \textbf{13} (1963), 524--536.

\bibitem{burgess-3}
D. A. Burgess, \emph{The character sum estimate with $r = 3$}, J. London Math. Soc. (2), \textbf{33} (1986), 219--226.

\bibitem{diamond}
H. Diamond, H. Montgomery, U. Vorhauer, \emph{Beurling primes with large oscillation}, Math. Ann. \textbf{334} (2006), no. 1, 1--36. 

\bibitem{emo}
A. Eskin, S. Mozes, H. Oh, \emph{On uniform exponential growth for linear groups}, Invent. Math. \textbf{160} (2005), no. 1, 1--30. 

\bibitem{elliott}
P. D. T. A. Elliott, H. Halberstam, \emph{A conjecture in prime number theory} Symp. Math. \textbf{4} (1968), 59--72.

\bibitem{fouvry} \'E. Fouvry, \emph{Autour du th\'eor\`eme de
    Bombieri-Vinogradov}, Acta Math. 152 (1984), no. 3-4,
  219--244.

\bibitem{ft}
\'E. Fouvry, \emph{Sur le probl\`eme des diviseurs de Titchmarsh}, J. Reine Angew. Math. \textbf{357} (1985), 51--76.

\bibitem{fi} \'E. Fouvry, H. Iwaniec, \emph{On a theorem of
    Bombieri-Vinogradov type}, Mathematika 27 (1980), no. 2, 135--152
  (1981).

\bibitem{fi-2} \'E. Fouvry, H. Iwaniec, \emph{Primes in arithmetic
    progressions}, Acta Arith. 42 (1983), no. 2, 197--218.

\bibitem{fik-3} \'E. Fouvry, H. Iwaniec, \emph{ The divisor function
    over arithmetic progressions.}  (With an appendix by Nicholas
  Katz.)  Acta Arith.  61 (1992), no. 3, 271--287.

\bibitem{FRI}
J. Friedlander, H. Iwaniec, Opera de cribro. American Mathematical Society Colloquium Publications, 57. American Mathematical Society, Providence, RI, 2010.

\bibitem{FKI}
\'E. Fouvry, E. Kowalski, H. Iwaniec, \emph{On the exponent of distribution of the ternary divisor function}, preprint.

\bibitem{gallagher}
P. X. Gallagher, \emph{Bombieri's mean value theorem}, Mathematika \textbf{15} (1968), 1--6.

\bibitem{gallagher-large}
P. X. Gallagher, \emph{A large sieve density estimate near $\sigma=1$}, Invent. Math. \textbf{11} (1970), 329--339.

\bibitem{gold}
L. Goldmakher, \emph{Character sums to smooth moduli are small}, Canad. J. Math. 62 (2010), no. 5, 1099--1115.

 \bibitem{graham} S. W. Graham, C. J. Ringrose, \emph{Lower bounds for
     least quadratic nonresidues}, Analytic number theory (Allerton
   Park, IL, 1989), 269--309, Progr. Math., 85, Birkh\"auser Boston,
   Boston, MA, 1990.

\bibitem{gs}
A. Granville, K. Soundararajan, \emph{The spectrum of multiplicative functions}, Ann. of Math. (2) \textbf{153} (2001), no. 2, 407--470.

\bibitem{gs0}
A. Granville, K. Soundararajan, \emph{An uncertainty principle for arithmetic sequences}, Ann. of Math. (2) \textbf{165} (2007), no. 2, 593--635. 

\bibitem{gs-new}
A. Granville, K. Soundararajan, \emph{Large character sums: Burgess's theorem and zeros of L-functions}, preprint.

\bibitem{ik}
H. Iwaniec, E. Kowalski, Analytic number theory. 
American Mathematical Society Colloquium Publications, 53. American Mathematical Society, Providence, RI, 2004.

\bibitem{kloost}
H. D. Kloosterman, \emph{On the representation of numbers in the form $ax^2 + by^2 + cz^2 + dt^2$}, Acta Mathematica \textbf{49} (1926), 407--464.

\bibitem{linnik}
U. V. Linnik, \emph{A remark on the least quadratic non-residue}, C. R. (Doklady) Acad. Sci. URSS (N.S.) \textbf{36} (1942), 119--120.

\bibitem{mv}
H. Montgomery. R. Vaughan, Multiplicative Number Theory I. Classical Theory.  Cambridge University Press, 2006.

\bibitem{motohashi} Y. Motohashi, \emph{An induction principle for the
    generalization of Bombieri's Prime Number Theorem}, Proc. Japan
  Acad. 52 (1976), 273--275.

\bibitem{polymath8a}
D.H.J. Polymath, \emph{New equidistribution estimates of Zhang type}, Algebra \& Number Theory \textbf{8-9} (2014), 2067--2199.

\bibitem{polymath8b}
D.H.J. Polymath, \emph{Variants of the Selberg sieve, and bounded intervals containing many primes}, Research in the Mathematical Sciences 2014, 1:12. 

\bibitem{postnikov}
A. G. Postnikov, \emph{On Dirichlet L-series with the character modulus equal to the power of a prime number}, 
J. Indian Math. Soc. (N.S.) \textbf{20} (1956), 217--226. 

\bibitem{rod}
K. A. Rodoskii, \emph{On non-residues and zeros of $L$=functions}, Izv. Akad. Nauk. SSSR Ser. Mat. \textbf{20} (1956), 303--306.

\bibitem{vinogradov} A. I. Vinogradov, \emph{The density hypothesis
     for Dirichlet L-series}, Izv. Akad. Nauk SSSR
  Ser. Mat. 29 (1965), 903--934.
		
\bibitem{vin}
I. Vinogradov,  Selected works.  With a biography by K. K. Mardzhanishvili. Translated from the Russian by Naidu Psv. Translation edited by Yu. A. Bakhturin. Springer-Verlag, Berlin, 1985.


\bibitem{wirsing}
E. Wirsing, \emph{Das asymptotische Verhalten von Summen \"uber multiplikative Funktionen II}, Acta Math. Acad. Sci. Hungar. \textbf{18} (1967), 411--467. 

\bibitem{zhang} Y. Zhang, \emph{Bounded gaps between primes}, Ann. of Math. (2) \textbf{179} (2014), no. 3, 1121--1174. 
  


\end{thebibliography}
\end{document}